\documentclass[a4paper, v2021]{lipics-v2021}
\usepackage{amsmath,amsthm,amssymb}
\usepackage{amsfonts}
\usepackage{calc}
\usepackage{amssymb}
\usepackage{amscd}
\usepackage{amsthm}
\usepackage[all]{xy}
\usepackage{tikz-cd}
\SelectTips{cm}{10}
\usepackage{graphics}
\usepackage{verbatim}
\usepackage{graphicx}
\usepackage{fancyhdr}
\usepackage{ushort}
\usepackage{enumitem}
\usepackage{color}
\usepackage{mathtools}

\usepackage{bm}
\usepackage{mathtools}
\usepackage[T1]{fontenc}
\usepackage[utf8]{inputenc}
\usepackage{lmodern}
\usepackage{xcolor}
\usepackage{url}
\usepackage{rotating}
\nolinenumbers
\theoremstyle{plain}

\newtheorem*{theorem*}{Theorem}

\theoremstyle{definition}

\title{The Topology of $k$-Robust Clique Complexes in Grid-like Graphs} 
\titlerunning{Topology of $k$-Robust clique Complexes} 

\author{Marek {Filakovský}}{Faculty of Sceince, Masaryk University, Kotlářská 267/2, 611 37 Brno, Czech Republic}{filakovsky@math.muni.cz}{https://orcid.org/0000-0001-5978-2623}{The work of the author was supported by the OP JAK Project MSCAFellow5\_MUNI (CZ.02.01.01/00/22\_010/0003229) and Masaryk University grant MUNI/SC/1957/2024 (GAMU MASH StG/CoG).}

\authorrunning{M. {Filakovský}} 


\newcommand{\AD}{\mathrm{AD}}{}
{}
\newcommand{\Cliq}[2]{\mathsf{Cliq}_{#1}(#2)}

\DeclarePairedDelimiter{\gr}{\lVert}{\rVert}

\DeclarePairedDelimiter{\floor}{\lfloor}{\rfloor}


\hideLIPIcs
\begin{document}
\maketitle

\begin{abstract}




We study $k$-robust clique complexes, a family of simplicial complexes that generalizes the traditional clique complex. Here, a subset of vertices forms a simplex provided it does not contain an independent set of size $k$.

We investigate these complexes for square sequence graphs, a class of bipartite graphs introduced here that are constructed by iteratively attaching $C_4$ cycles. This class includes rectangular grid graphs $G_{m,n}$. We show that for $k=2$ and $k=3$, the homotopy type is a wedge sum of $(2k-3)$-dimensional spheres, a result we extend to arbitrary $k$ under specific structural constraints on the attachment sequence.

Our approach utilizes K\"{o}nig's theorem to decompose the complex into manageable components, whose homotopy types are easy to understand. This then enables an inductive proof based on the decomposition and standard tools of algebraic topology. Finally, we utilize Alexander duality to connect our results to the study of total-$k$-cut complexes, generalizing recent results concerning the homotopy types of total-$k$-cut complexes for grid graphs.



\end{abstract}
\newpage
\section{Introduction}
Simplicial complexes are algebraic structure that are well-suited for capturing the hereditary properties of families of objects. Various complexes have been introduced to study graph-theoretic properties; notably, \emph{neighbourhood complexes} were used by Lov{\'a}sz~\cite{Lovasz78} to prove the Kneser Conjecture, \emph{independence} complexes have been studied extensively in connection to physics~\cite{fendleySV2003} and \emph{clique complexes} have found applications in topological data analysis~\cite{GiustiPCI2015,ReimannNSLMH2017} and sensor networks~\cite{desilvaG2007}. A comprehensive study of graph simplicial complexes is provided by Jonsson~\cite{Jonsson08}, though many new families were since introduced (see e.g. \cite{BayerDJRSX24, white2016} for examples in relation to commutative algebra and chromatic polynomials).


Simplicial complexes provide a bridge between combinatorics, algebra and topology:
In combinatorics, one properties such as shellability, partitionability and $f$-vectors. Topologically, they serve as discrete models of a CW-complexes and enabling the computation of invariants compute topological invariants such as homology and homotopy type. Finally, the connection to algebra is facilitated by the Stanley-Reisner ring~\cite{reisner1976} where algebraic properties such as Cohen-Macaluay and Gorenstein property are  linked to the topological (and combinatorial) properties of the underlying simplicial complex~\cite{reisner1976,bjorner1980}. 

In this paper, we study the $k$-robust clique complex $\Cliq{k}{G}$. A subset of vertices $W\subseteq V(G)$ forms a simplex in $\Cliq{k}{G}$ if the induced subgraph $G[W]$ does not contain an independent set of size $k$. These complexes generalize the standard clique complexes (which corresponds to the case $k=2$).
Our naming is motivated by the relaxation of the clique condition: we "accept" a vertex subset if it is sufficiently dense, permitting missing edges so long as they do not form a forbidden sparsity pattern (an independent set of size $k$). Robust clique complexes were originally introduced in~\cite{KimL2022} (there, the authors did not give the complex a name), where they were chiefly studied in the context of their collapsibility number.
We continue the study of $k$-robust clique complexes by focusing on \emph{grid graphs} $G_{m,n}$ (i.e. Cartesian products of paths with $(m-1)$ and $(n-1)$ edges, respectively). Our main result characterizes their homotopy type as follows:

\begin{theorem}\label{thm:main}
Let $G_{m,n}$ be an $m \times n$ grid graph with $m, n \geq 2$. For  $ k\in \{2,3\}$,
\[
 \gr{\Cliq{k}{G_{m,n}}} \simeq \bigvee_{
    \binom{(m-1)(n-1)}{k-1}} S^{2k-3}.
\]
\end{theorem}

Theorem~\ref{thm:main} is obtained as a consequence of a more general result concerning square sequence graphs (Definition~\ref{d:sqseq}). Intuitively, these are bipartite graphs constructed by iteratively attaching $C_4$ cycles via "edge gluing" or "corner gluing" (see Figure~\ref{fig:attach}). Our core technical result provides a recursive characterization:

\begin{theorem}\label{thm:main2}   
Let $G$ be a square sequence graph with sequence $C_4 \cong H_1 \subset  H_2 \subset \cdots \subset H_n \cong G$, $n\geq 1 $. For $k = 3$,
\[
\gr{\Cliq{k}{H_n}} \simeq \bigvee_{\gamma_n} S^{3}, 
\]
where $\gamma_n \geq 0$ is determined by a recurrence relation based on the homotopy type of $\Cliq{k}{H_{n-1}}$ and the specific attachment of the $n$-th square.
\end{theorem}

By specializing the attachment mode, we obtain the following corollary:

\begin{corollary}\label{c:edgeseqeunce}
Let $G$ be a square sequence graph with sequence
$
G_1 = H_1 \subseteq H_2 \subseteq \cdots \subseteq H_{n} \cong G
$
obtained solely via edge gluing. For $n\geq 1$, $k\geq 2,$
\[
    \gr{\Cliq{k}{H_n}} \simeq \bigvee_{\binom{n-1}{k-1}} S^{2k-3}.
\]
\end{corollary}

Our study of robust clique complexes is further (dually) motivated by the study of total $k$-cut complexes $\Delta^t_k(G)$ introduced in~\cite{BayerDJRSX24} to generalize results of Fröberg~\cite[Theorem~1]{Froberg1990} and Eagon and Reiner~\cite[Proposition~1]{EagonReiner1998} concerning an interplay between graph chordality, the clique complex and $2$-linear resolutions.

 In a series of subsequent works~\cite{BayerDMJRSX24,BayerDMJSX25,ChandrakarHRS24,ShenSYZZ25}, the topological and combinatorial properties of $\Delta^t_k(G)$ were investigated across several graph classes. In particular for grid graphs, in~\cite[Theorem~4.16]{BayerDJRSX24}, the authors determined 
 \[
\gr{\Delta_2^t (G_{m,n})} \simeq \bigvee_{{(m-1)(n-1)}} S^{mn-4} \quad m,n\geq2,
\]
while more recent results~\cite[Theorem~1.1]{ChandrakarHRS24} established
\[
\gr{\Delta_k^t (G_{2,n})} \simeq \bigvee_{\binom {n-1}{k-1}} S^{2n-2k} \quad n,k\geq2; \qquad \gr{\Delta_3^t (G_{3,n})} \simeq \bigvee_{\binom {2n-2}{2}} S^{3n-6}\quad n\geq2.
\]

The key observation for our work is that $k$-robust clique complexes $\Cliq{k}{G}$ are \emph{precisely the Alexander duals} of total $k$-cut complexes $\Delta_k^t(G)$. This implies that their homotopy types are closely linked\footnote{From the perspective of stable homotopy theory, the \emph{stable} homotopy type of one determines the \emph{stable} homotopy type of the other.}, and can be derived from one another. For details on Alexander duality, see~\cite{BjornerT09} and Theorem~\ref{thm:duality}. We showcase the power of this duality by generalizing ~\cite[Theorem~1.1]{ChandrakarHRS24} to arbitrary grids as a corollary of Theorem~\ref{thm:main}:
\begin{corollary}\label{thm:TotalCutMain}
Let $G_{m,n}$ be a  $m \times n$ grid graph, $m, n \geq 2$ and $k\in \{2,3\}$. Then
\[
\gr{\Delta_k^t (G)} \simeq \bigvee_{
    \binom{(m-1)(n-1)}{k-1}} S^{mn - 2k}.
\]
\end{corollary}

We remark that the primary advantage of $\Cliq{k}{G}$ lies in their low dimensionality. This makes the complexes more combinatorially tractable than the associated cut complexes, allowing for an inductive analysis that avoids the more involved combinatorial methods required in~\cite{ChandrakarHRS24}.

Naturally, we ask whether Theorem~\ref{thm:main} holds for $k > 3$. In general, the specific enumerative formula does not extend, as showcased in Example~\ref{ex:gridformula}. However, we propose the following generalization of Theorem~\ref{thm:main2}:
\begin{conjecture}
For all $k\geq2,n\geq 1$, the $k$-robust clique complex $\Cliq{k}{G}$ of graph $G$ with a square sequence $C_4 \cong H_1 \subset  H_2 \subset \cdots \subset H_n \cong G$ is homotopy equivalent to a wedge sum of ($2k-3$)-dimensional spheres. Further, the number of spheres  is determined by a recurrence relation based on the homotopy type of $\Cliq{k}{H_{n-1}}$ and the specific attachment of the $n$-th square.
\end{conjecture}
We suspect that the specific formula in Theorem~\ref{thm:main} remains viable when $k$ is sufficiently small relative to the independence number $\alpha(G_{m,n})$. Our current technical bottleneck involves the complexity of the recursive step in the proof of Theorem~\ref{thm:main2} for $k > 3$ and corner gluing. Finally, we note that if $m=1$ or $n=1$, the grid graph $G_{m,n}$ reduces to a path. Here, the topology of total cut complexes of paths were determined in~\cite{BayerDJRSX24}, and the topology of robust clique complexes can be deduced via Alexander duality.

We now present a shorthand summary of the main steps of the proof of Theorem~\ref{thm:main2} and its corollaries. Full proofs are postponed until Section~\ref{s:Proofs} and we provide necessary preliminaries in Section~\ref{s:Preliminaries}. 

\subsection*{The proof of Theorem~\ref{thm:main2} in a nutshell:}
The proof of Theorem~\ref{thm:main2} proceeds by induction on the length of the square sequence. The case $k=2$ is straightforward (Lemma~\ref{l:clique}) and serves as the base for the  inductive step for $k=3$.

We assume the homotopy type of $\Cliq{k}{H_i}$ is known for $1 \leq i < n$ and compute the homotopy type of $\Cliq{k}{H_{i+1}}$. To this end, we first utilize Lemma~\ref{l:decomposition}. This Lemma, based on a  variant of K\"{o}nig's Theorem and Berge's Lemma (Lemma~\ref{l:edge_removal}), allows us to characterize $\Cliq{k}{H_{i+1}}$ as the union of two subcomplexes:
\[
\Cliq{k}{H_{i+1}} = K \cup L, \quad \text{where } K = \Cliq{k}{H_i} \text{ and } L = \Cliq{k-1}{H_i} \oplus \Cliq{2}{C_4}.
\]
Here, $\oplus$ denotes the \emph{embedded join} (Definition~\ref{d:join}). While this construction has appeared in algebraic contexts~\cite{SimisU2001,ChoeJ2025}, from which we borrow the name\footnote{Alternatively, one can say that it is a "simplex-wise union".}, we believe this might be its first appearance in the combinatorics of simplicial complexes. For us, it is important that the embedded join behaves up to homotopy like the standard join if the intersection of the factors is a full simplex (see Definition~\ref{d:join} and Lemma~\ref{l:joinsimplex} for details).

Inductively, we assume that $K$ has the homotopy type of a wedge sum of $3$-spheres, and we analyze the homotopy type of this union based on the intersection $K \cap L$:
\begin{itemize}
    \item \textbf{Edge Attachment:} When $C_{4}$ is attached along an edge, Lemma~\ref{l:joinsimplex} shows that the intersection $K \cap L$ is contractible. Consequently, by Lemma~\ref{l:contractionSuspension}, $K\cup L$ is homotopy equivalent to the wedge sum $K \vee L$. 
    Because we are edge gluing, Lemma ~\ref{l:joinsimplex} implies that embedded join has the same homotopy as the standard join, hence 
    \[
    \gr{L} \simeq \gr{\Cliq{k-1}{H_i} \star \Cliq{2}{C_4}} \simeq \gr{\Cliq{k-1}{H_i} \star S^1}.\] 
    Since $k = 3$, $\Cliq{k-1}{H_i}$ has the homotopy type of a wedge sum of spheres $S^1$ by Lemma~\ref{l:clique} and thus $\gr{L}\simeq \bigvee_\beta S^3$ for some $\beta\geq 0$.
    \item \textbf{Corner Attachment:} When $C_{4}$ is attached along a path of length $2$, we use a slightly different decompostion: $K$ remains the same while \[ L = \{\sigma \in \Cliq{k}{H_i} \mid \sigma \cup \{x\} \in \Cliq{k}{H_i}\}. \] One can now see that $L$ is topologically simpler - it is a simplicial cone $L \cong C(K\cap L)$. Further, we can show that $K \cap L$, while not contractible itself, is still \emph{contractible as a subcomplex} of $K$. Then Lemma~\ref{l:contractionSuspension}, gives $ \gr{K\cup L} \simeq \gr{K \vee \Sigma(K \cap L)}$. Since this is not sufficient, we use the fact $k = 3$ and study the homotopy type of $(K \cap L)$ using a homotopy pushout--style argument and we are  able to show that $\gr{K\cap L} \simeq \Sigma (\bigvee_\beta S^1)$, hence  $\gr{\Cliq{k}{H_{i+1}}} \simeq  \gr{K} \vee \Sigma^2 (\bigvee_\beta S^1)\simeq \gr{K} \vee \bigvee_\beta S^3$.
\end{itemize}
 
To prove Theorem~\ref{thm:main}, we apply Theorem~\ref{thm:main2} to a chosen square sequence of the grid $G_{m,n}$. We notice that the recursive formula for computing the homotopy type of $\gr{\Cliq{k}{H_{i+1}}}$ is the same regardless of the type of attachment and the result then follows from the binomial formula.

For Corollary~\ref{c:edgeseqeunce}, we observe that the argument in the proof of Theorem~\ref{thm:main2} for edge attachments never really uses $k=3$: the homotopy type of $\Cliq{k-1}{H_{i}}$ can be taken to be a wedge sum of spheres $S^{2(k-1) - 3}$ via the induction assumption.

Finally, to prove Corollary~\ref{thm:TotalCutMain}, we use Theorem~\ref{thm:main} and apply Alexander Duality (Theorem~\ref{thm:duality}). The duality implies that the (co)homology of the total cut complex $\Delta_k^t$ is concentrated in dimension $mn-2k$. Thus $\Delta_k^t$ has the same \emph{(co)homology} type as a wedge sum of spheres $S^{mn-2k}$. Since $\Delta_k^t$ is further simply connected, which we show by a simple combinatorial argument, a version of the Whitehead Theorem (Theorem~\ref{t:whitehead}) confirms that $\Delta_k^t$ is homotopy equivalent to a wedge of spheres.

\section{Preliminaries}\label{s:Preliminaries}
This section is divided into two parts. In the first part, we recall basic definitions concerning graphs and simplicial complexes, define robust clique complexes, square sequence graphs and we conclude the subsection with two technical lemmas regarding decompositions of robust clique complexes.

In the second part, we summarize essential notions from algebraic topology required in Section~\ref{s:Proofs}. Most of the material presented here is standard; for a more comprehensive treatment, we refer the reader to~\cite{Jonsson08,Sti12, Hatcher2002,Kozlov2008CAT} and~\cite{matousek2003}.

\subsection*{Discrete notions}
By graph we always mean a \emph{simple graph} with no multiedges or loops.
\begin{definition}\label{d:graph}
Let $G(V,E)$ be a graph, $W \subseteq V$. By $G[W]$ we denote the induced graph on the subset of vertices, i.e. $V(G[W]) = W$ and $\{v,w\} \in E(G[W])$ if $\{v,w\}\in E$ and $v,w\in W$.
\begin{itemize}
    \item We call $W$ a clique if $G[W]$ is the complete graph , i.e. an edge connects any two vertices of $W$.
    \item We call $W$ an independent set of size $|W|$ if $G[W]$ contains no edges. The set of independent sets of size $k$ is denoted $\mathcal{I}_k$.
    \item We call the size of the largest independent set in $G$ the \emph{independence number of} $G$ and denote it by $\alpha(G)$.
\end{itemize}
\end{definition}

We now formally define grid graphs: 
\begin{definition}[Grid graph]
    Let $m,n \in \mathbb{N}$. The grid graph $G_{m,n}$ is defined as follows: 
    \begin{align*}
    V(G_{m,n}) &= \{(i,j) \mid 1\leq i\leq m,  1\leq j\leq n \},\\    
    E(G_{m,n}) &= \{(i,j), (i,j+1) \mid 1\leq i\leq m,  1\leq j\leq (n-1) \}\\        
    &\cup \{(i,j), (i+1,j) \mid 1\leq i\leq (m-1),  1\leq j\leq n \}.
\end{align*}
\end{definition}
In this paper, we always consider a path not to be a grid graph, i.e. $m, n\geq2$.

We now introduce the family of \emph{square sequence graphs}:
\begin{definition}\label{d:sqseq}
We call a graph $G$ a \emph{square sequence graph of length $n$}, if there exists a sequence of subgraphs called the \emph{square sequence}
\[
G_1 = H_1 \subseteq H_2 \subseteq \cdots \subseteq H_n \cong G
\]
and subgraphs $G_i\subseteq H_i$ such for all $1\leq i \leq n$, the following conditions hold: 
    \begin{enumerate}
        \item $G_i \cong C_4, 1\leq i \leq n$. 
        \item $H_{i+1} = H_{i} \cup G_{i+1}$ for all $1\leq i < n$;
        \item $H_{i} \cap G_{i+1}$ is isomorphic to a path of length $1$ (We call this \emph{edge attachment/gluing}) or a path of length $2$ (We call this \emph{corner attachment/gluing}). 
    \end{enumerate}
\end{definition}

Whenever we analyse a situation when a square $G_i$ is being attached to $H_{i-1}$, we denote the vertices $V(G_i) = \{x,y,u,v\}$ and  $E(G_i) = \{\{x,u\}, \{x,v\}, \{y,u\},\{y,v\}\}$. For edge gluing we assume $E(H_{i-1} \cap G_{i}) = \{\{u,y\}\}$ and for corner gluing $E(H_{i-1} \cap G_{i}) = \{\{u,y\} , \{v,y\}\}$, see Figure~\ref{fig:attach} for illustration. We believe this will not cause confusion since we never mix multiple attachments at the same time.

\begin{figure}[ht]
    \centering
    \includegraphics[width=0.4\textwidth]{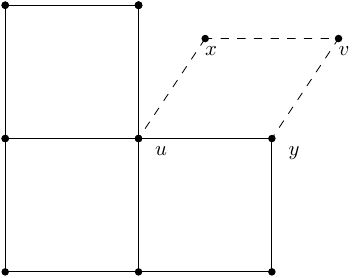} 
    \includegraphics[width=0.4\textwidth]{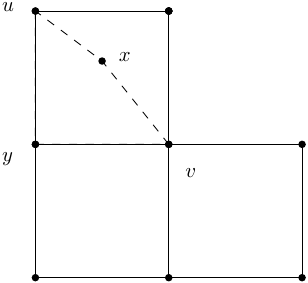}
    \caption{Edge and corner gluing}
    \label{fig:attach}
\end{figure}

Since our main object of study are simplicial complexes, we briefly remind of a number of basic definitions and constructions. For further details, we refer to~\cite[Section 1.3 - 1.5]{matousek2003}:

Let $V$ be a finite set, Then $K \subseteq 2^V$ is called a \emph{(finite, abstract) simplicial complex} if  $\{v\} \in K, \forall v\in V$ and if it is downward closed under inclusion: $B\subseteq A \in K$ implies $B \in K$. The elements $A\in K$ are called \emph{simplices}, the \emph{dimension} $\dim(A)$ of a simplex $A\in K$ is $\dim(A) = (|A| - 1)$, where $|A|$ is the number of elements of $A$ and the \emph{dimension} of $K$ is $\dim(K) = \max_{A\in K}\{\dim(A)\mid A\}$. The set $V$ is the \emph{vertex set} of $K$ and if $K = 2^V$, we say that $K$ is a $|V|-1$-simplex and use the notation $K \cong \Delta^{|V|-1}$.

A \emph{simplicial map} between simplicial complexes $K$ and $L$ is the map of their vertex sets that preserves simplices $A\in K \Rightarrow f(A) \in L$.


We now define the main object of our study, originally introduced in~\cite{KimL2022}:
\begin{definition}\label{def:myClique}
For $k\geq 2$, the robust clique complex $\Cliq{k}{G}$, of a graph $G$ is 
\[
\Cliq{k}{G} = \{ W \subseteq V(G) \mid W \not \supseteq \sigma, \forall \sigma \in \mathcal{I}_k(G)\}.
\]
\end{definition}
We remark that if $k>\alpha(G)$, then $\Cliq{k}{G} = 2^{V(G)}$. 

\begin{figure}[ht] 
    \centering
    \includegraphics[width=0.6\textwidth]{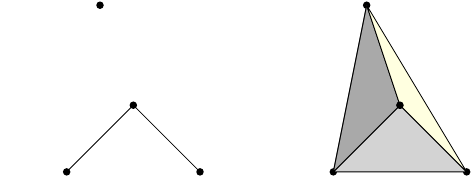} 
    \caption{The robust $3$-clique complex of the graph on the left consists of three $2$-simplices.}
    \label{fig:ex_cliq}
\end{figure}

\begin{example}\label{ex:gridformula}
Consider $\Cliq{6}{G_{4,3}}$. The independence number of $G_{4,3}$ is $6$ and one can see that there are exactly two disjoint independent sets $A,B$ of size $6$ in $G_{4,3}$. The complex $\Cliq{6}{G_{4,3}}$ consists exactly of those subsets of vertices that do not contain $A$ or $B$. These can be characterized as joins of the boundaries of two $5$-simplices $\partial \Delta^5 \star \partial \Delta^5$, which is homotopy equivalent to $S^4 \star S^4 \simeq S^9$ and thus not a wedge of $6$ copies of $S^9$ as suggested by the formula in Theorem~\ref{thm:main}.
\end{example}

Direct computation of the homotopy type of $\Cliq{k}{G_{m,n}}$ as in the example above is always available when $k$ is equal to the independence number $\alpha(G_{m,n})$ of the grid graph $G_{m,n}$ and the result is either a point (if $mn$ is odd) or a single sphere ($mn$ even). This follows from the fact that there are either one or two maximum independent sets in $G_{m,n}$ depending on the parity of $m$ and $n$.

\begin{definition}
The \emph{combinatorial Alexander dual} of a simplicial complex $K$ with vertex set $V$ is a simplicial complex $K^\AD$ on the same vertex set with 
$K^\AD = \{ W \subseteq V \mid V \setminus W \not\in K
\}$.
\end{definition}
In~\cite{BayerDJRSX24}, the authors defined the  total cut complex $\Delta_k^t (G)$ as
\[
\Delta_k^t (G) = \{ W\subseteq V(G) \mid \exists~\sigma \in \mathcal{I}_k, \sigma \subseteq (V(G)\setminus W)\}.
\]
Hence the Alexander dual of a total $k$-cut complex is exactly the $k$-robust clique complex: ${\Delta_k^t (G)}^\AD = \Cliq{k}{G}$.




In our computations, we will make use of the following join and embedded join constructions for simplicial complexes.
\begin{definition}\label{d:join}
The \emph{join} of simplicial complexes $K$ and $L$ is the complex $K \star L$ with
$$V(K \star L) = V(K)\times \{0\} \cup V(L)\times\{1\}$$
$$\sigma \in K \star L \iff \sigma = \alpha\times\{0\} \cup \beta\times\{1\} \quad \text{where } \alpha \in K, \beta \in L.$$
their \emph{embedded join} is given by
$$V(K \oplus L) = V(K) \cup V(L)$$
$$\sigma \in K \oplus L \iff \sigma = \alpha \cup \beta \quad \text{where } \alpha \in K, \beta \in L.$$
\end{definition}
Clearly, join and embedded join coincide if $V(K) \cap V(L) = \emptyset$. We will see later that in case the intersection of the two complexes is nice (e.g. a simplex), the embedded join behaves up to homotopy like the standard join (Lemma~\ref{l:joinsimplex}).

 The join $K\star a$ of a simplicial complex $K$ and a $0$-simplex $a$ is called a \emph{cone} and denoted $CX$ and the join of $K$ with two $0$-simplexes $a$,$b$ is called a \emph{suspension} and denoted $\Sigma X$.



We are now ready to state two technical lemmas. The first follows from classical results of König and Berge applied in bipartite graphs.
\begin{lemma}[An application of König's theorem]\label{l:edge_removal}
Let $k \geq 2$, let $G$ be a bipartite graph and let $W \in \Cliq{k+1}{G}$ such that $W \not\in \Cliq{k}{G}$. Let $M$ be a maximal matching on $G[W]$. Then for every edge $e \in M$, $W \setminus e \in \Cliq{k}{G}$.
\end{lemma}
\begin{proof}
We begin with the fact that in bipartite graphs, the following equality (Gallai's formula) holds: 
\[
\alpha(G[W]) + \tau(G[W]) = |W|.
\]
here $\tau(G[W])$ is the size of a minimum vertex cover\footnote{The vertex cover of a graph is set of vertices $C \subset V(G)$ such that any edge in $G$ contains a vertex in $C$.} of $G[W]$. We know that $\alpha(G[W]) = k$ by our assumptions on $W$.

Since $G$ is bipartite, $G[W]$ is also bipartite and König's theorem~\cite{konig1916} implies that $\tau(G[W])$ is equal to the size of the maximum matching in $G[W]$. Thus let $M \subseteq E(G[W])$ be a maximum matching in $G[W]$. By Berge's Lemma~\cite{berge1957}, for any an edge $e = \{u,v\} \in M$, $M' = M \setminus e$ is a maximum matching in $G' = G[W \setminus e]$. Using the equality above again and substituting gives us
\begin{align*}
\alpha(G') + \tau(G') &= |W| - 2. \\
\alpha(G') + \tau(G[W]]) - 1 &= |W| - 2. \\
\alpha(G') = \alpha(G[W]) - 1 &= k-1.
\end{align*}
Hence $G'$ does not admit an independent set of size $k$ and $W \setminus e  \in \Cliq{k}{G}$.
\end{proof}

Finally, we present the decomposition Lemma for robust clique complexes of square sequence graphs:

\begin{lemma}[Decomposition of a  $k$-robust clique complex of a square sequence graph]\label{l:decomposition}
   Let $C_4 \cong H_1 \subseteq \cdots \subseteq H_n \cong G$ be a square sequence of a graph $G$ and let $k>2$. Denote $K = \Cliq{k}{H_{n-1}}$ and $L = \Cliq{k-1}{H_{n-1}} \oplus \Cliq{2}{G_{n}}$. Then 
   \begin{align*}
   \Cliq{k}{H_{n}} = \Cliq{k}{H_{n-1}} \cup (\Cliq{k-1}{H_{n-1}} \oplus \Cliq{2}{G_n}) = K \cup L.
    \end{align*} 
Further,
\begin{align*}
K \cap L = \Cliq{k-1}{H_{n-1}} &\oplus \Cliq{2}{G'_{n}},
\end{align*}
where $G'_{n} = G_n[V(G_{n}) \setminus \{x\}]$ in case of corner gluing and $G'_{n} = G_{n}[ V(G_n) \setminus \{x,v\}]$ in case of edge gluing.

\end{lemma}
\begin{proof}
$\supseteq$: $\Cliq{k}{H_{n}}$ contains $\Cliq{k}{H_{n-1}}$ since $H_{n-1} \subseteq H_n$. Next, let $W \in \Cliq{k-1}{H_{n-1}}$. Since $W$ contains no independent set of size $k-1$, $W \cup e$, where $e\in E(H_{n}) \supseteq E(G_n)$, does not contain an independent set of size $k$.

$\subseteq$: Let $W\in \Cliq{k}{H_{n}}$ be such that $W \not\in\Cliq{k}{H_{n-1}}$ (otherwise we are done) and let $\alpha(G[W]) = \ell \geq (k-1)$.
We distinguish whether $G_n$ is attached to $H_n$ by corner gluing or edge gluing. We remind the reader, that in both cases, $x$ is the vertex of $G_n$ that is not contained in $H_{n-1}$ and $v$ is the vertex of $G_n$ that is contained in $H_{n-1}$ in case of corner gluing and not contained in $H_{n-1}$ in case of edge gluing. In any case, the proof proceeds by applying Lemma~\ref{l:edge_removal} to $W$ and showing that the resulting set is contained in $\Cliq{k-1}{H_{n-1}}$.

\textbf{$G_n$ is attached to $H_{n-1}$ by corner gluing:} Since $W \not\in\Cliq{k}{H_{n-1}}$, we must have $x\in W$. If $u,v \not\in W$, then $x$ must be contained in every largest independent subset of $W$, hence $\alpha(G[W \setminus \{x\}]) = \ell-1$ and it follows that $W \setminus \{x\} \in \Cliq{k-1}{H_{n-1}}$. If $W$ contains $x$ and either $u$ or $v$ (or both), then one of the edges $\{u,x\},\{v,x\}$ is a part of a maximal matching on $H_n[W]$. Without a loss of generality, let it be $\{u,x\}$. By Lemma~\ref{l:edge_removal}, we have $\alpha(G[W \setminus \{u,x\}]) < \alpha(G[W])$, thus $W \setminus \{u,x\} \in \Cliq{k-1}{H_{n-1}}$.

\textbf{$G_n$ is attached to $H_{n-1}$ by edge gluing:}
Since $x,v \not \in V(H_{n-1})$ and $W \not\in\Cliq{k}{H_{n-1}}$, at least one of the vertices $x$ or $v$ must be in $W$. If both are in $W$, then $W \setminus \{v,x\} \in \Cliq{k-1}{H_{n-1}}$ by Lemma~\ref{l:edge_removal}. Let  only one of $x,v$ be in $W$ and suppose without a loss of generality that it is $x$. If $u\in W$, the edge $\{x,u\}$ is in a maximum matching in $W$ thus removing $\{x,u\}$ from $W$ gives us a set in $\Cliq{k-1}{H_{n-1}}$ by Lemma~\ref{l:edge_removal}. If $u\not\in W$, $x$ is disconnected in $H_{n-1}[W]$, therefore $x$ belongs to every independent set of maximal size in $W$. Removing $x$ from $W$ then gives us a set in $\Cliq{k-1}{H_{n-1}}$.

Finally, the intersection  follows from the definition of the $k$-clique complex and the fact that $G'_{n} = G_n[V(G_{n}) \setminus \{x\}] \cap H_n $ in case of corner gluing and $G'_{n} = G_{n}[ V(G_n) \setminus \{x,v\}] \cap H_n$ in case of edge gluing.
\end{proof}

\subsection*{Topological preliminaries}
To make the main ideas of the proof of Theorem~\ref{thm:main2} eligible, we remind of basic results from algebraic topology. However some details including the rudimentary definitions such as topological space, continuous map, $n$-sphere $S^n$ and $n$-disc $D^n$ are omitted. For full details, we refer the reader to textbooks~\cite{Hatcher2002,Kozlov2008CAT} and~\cite{matousek2003}.

We first define a geometric realization of a simplicial complex. Intuitively, one can interpret the simplices of a simplicial complex $K$ as geometric objects, i.e. $A\in X$ can be seen as a $\dim(A)$-simplex: a convex hull of $d = \dim(A) + 1$ affine independent points in $\mathbb{R}^d$. Thus $0$-simplex is a point, $1$-simplex a line segment $2$-simplex a triangle etc. The geometric realization $\gr{K}$ is then obtained by gluing the simplices using the data from $K$. 

The intuitive definition can be made precise, however, we use the following definition that is shorter and remark that both definitions coincide. Further details can be found in the short exposition~\cite{nLabscpx} and in \cite[Section~1.3]{matousek2003}.

\begin{definition}
Let $K$ be a simplicial complex on the vertex set $V(K)$, where $d = |V(K)|$. We define a topological space $\gr{K}$ called the \emph{geometric realization} of $K$ as a set of functions $\alpha \colon V(K) \to [0,1]$, to the closed interval $I = [0,1]$, such that
\begin{itemize}
    \item if $\alpha \in \gr{K}$, the set $\{V(K), \alpha(v)\neq 0\}$ is a simplex of $K$.
    \item for each $\alpha \in \gr{K}$,  $\sum_{v \in V(K)} \alpha(v) = 1$.
\end{itemize}
\end{definition}
The simplicial map of simplicial complexes $K \to L$ further induces continuous map of their geometric realizations in an obvious way. In this text we frequently do not distinguish between simplicial complexes and their geometric realizations, as it should be clear from the context.

In this text, every space and especially every geometric realization of a simplicial complex will be a CW complex. A \emph{CW complex} is a space $X$ together with a increasing sequence of subspaces (called a \emph{filtration}) 
\[
  X_0 \subseteq X_1 \subseteq X_2 \subseteq \cdots\subseteq X,
\]
with the following properties: $X_0$ is a discrete set of points (called \emph{vertices} or \emph{$0$-dimensional cells}) and $X_{i+1}$ is constructed by attaching a set of $(i+1)$-dimensional discs $D^{i+1}_\alpha$ to $X_i$ along their boundary $\partial D^{i+1}$ via continuous maps $g_\alpha \colon \partial D^{i+1} _\alpha = S^{i} _\alpha \to X_i$. Thus
\[
  X_{i+1} = (X_i \cup \coprod_\alpha D^{i+1}_\alpha) / \sim
\]
where $\sim$ identifies $g_\alpha(x)\in X_i$ with $x\in \partial D^{i+1}_\alpha$. Finally, the topology on $X = \bigcup_n X_n$ is the so-called \emph{weak topology} (i.e., a set $U\subseteq X$ is open if and only if $X\cap X_i$ is open in $X_i$ for every $i$). The subspace $X_i$ is called the $i$-dimensional skeleton of $X$.

\begin{definition}[Homotopy - basic notions]
Let $f,g\colon X \to Y$ be continuous maps of topological spaces. we say that $f$ and $g$ are \emph{homotopic} and denote this by $f \sim g$ if there exists a continuous map $H \colon [0,1] \times X \to Y$ such that $H|_{\{0\}\times X} = f$ and $H|_{\{1\}\times X} = g$.

We say that spaces $X,Y$ are homotopy equivalent (or that $X,Y$ have the same \emph{homotopy type}) and denote this $X \simeq Y$ if there are maps $f\colon X\to Y$ and $g \colon Y \to X$ such that $\mathsf{id}_X \sim gf$ and $\mathsf{id}_Y \sim fg$. Further $f,g$ are called homotopy equivalences. 

By $[f]$ we denote the \emph{homotopy class}- a set of maps $f'$ with $f' \sim f$ and by 
\[[X,Y] = \{[f] \mid f\colon X \to Y\},\] we denote \emph{the set of homotopy classes of maps}.
\end{definition}

We call a space $X$ \emph{path connected} if for every $x_0, x_1 \in X$ there exists a continuous map $H \colon [0,1] \to X$ such that $H(0) = x_0$ and $H(1) = x_1$.

A space $X$ is further \emph{contractible} if it is homotopy equivalent to a point, or more precisely, there is a homotopy equivalence between $\mathsf{id}_X$ and a constant map to some $x\in X$.

A subspace $A\subseteq X$ is contractible in $X$ if the inclusion map $i\colon A \to X$ is homotopic to a constant map.

Given two spaces (CW complexes) $X$, $Y$ and points $x\in X_0, y\in Y_0$, their \emph{wedge sum} is $X \vee Y$ is $X\sqcup Y/\sim$, where $x_0\sim y_0$. If both $X$ and $Y$ are path connected, the homotopy type of the space $X \vee Y$ does not depend on the choice of $x_0, y_0$.

\subsubsection*{Homotopy groups, (co)Homology and Whitehead theorem}

The $n$-th homotopy group of a path connected space $X$ is the set $\pi_n(X) = [S^n, X]$ i.e. the set of homotopy classes of continuous maps from the $n$-sphere $S^n$ to $X$. The set $\pi_0(X)$ is the set of connected components\footnote{Here we do not require $X$ to be path connected.} of $X$ and $\pi_1(X)$ is the \emph{fundamental group} of $X$. We call a path connected CW complex $X$ \emph{simply connected} if $\pi_1(X) = 0$. We call a map $f\colon X \to Y$ a \emph{weak homotopy equivalence} if the induced map $\pi_n(f)\colon \pi_n(X) \to \pi_n(Y)$ is an isomorphism for all $n\geq 0$.

When dealing with CW complexes, we can utilise the following result by Whitehead that says that for CW complexes weak homotopy equivalence induces homotopy equivalence 
\begin{theorem}
    Let $f\colon X \to Y$ be a weak homotopy equivalence of CW complexes. Then $f$ is a homotopy equivalence.
\end{theorem}

We will further make use of another invariant of topological spaces called homology and cohomology groups, denoted $H_n(X; \mathbb{Z})$ and $H^n(X; \mathbb{Z})$ for $n\geq 0$, respectively.
The (co)homology groups are Abelian groups that can be defined in several equivalent ways, e.g. via simplicial, singular or cellular homology. We refer to~\cite[Chapter~2]{Hatcher2002} for full details. One can further define the reduced homology and cohomology groups $\tilde{H}_n(X; \mathbb{Z})$ and $\tilde{H}^n(X; \mathbb{Z})$ for $n\geq 0$ that are the same as the usual homology and cohomology groups for $n>0$ and differ only in dimension $0$ where $\tilde{H}_0(X; \mathbb{Z}) = H_0(X; \mathbb{Z})/\mathbb{Z}$ and $\tilde{H}^0(X; \mathbb{Z}) = \ker(H^0(X; \mathbb{Z}) \to H^0(*; \mathbb{Z}))$.

For our purposes, it is enough to know that (reduced) homology and cohomology groups are homotopy invariants of topological spaces, and are "weaker" than weak homotopy equivalences i.e. if $\pi_n(X)\simeq \pi_n(Y)$, then $H_n(X; \mathbb{Z}) \cong H_n(Y; \mathbb{Z})$ and $H^n(X; \mathbb{Z}) \cong H^n(Y; \mathbb{Z})$ for all $n\geq 0$. Further, a continuous map $f\colon X \to Y$ induces homomorphisms on homology and cohomology groups, denoted $f_*\colon H_n(X; \mathbb{Z}) \to H_n(Y; \mathbb{Z})$ and $f^*\colon H^n(Y; \mathbb{Z}) \to H^n(X; \mathbb{Z})$, respectively.

We will make use of the following consequence of the Whitehead theorem and theorem of Hurewicz that is sometimes referred to as Homology Whitehead theorem that shows that even an equivalence on the weaker invariant called homology can induce a homotopy equivalence:
\begin{theorem}[\cite{Hatcher2002}, Theorem~4.5]\label{t:whitehead}
Let $f: X \to Y$ be a continuous map between two simply connected CW complexes. If the induced map on integral homology groups,
$$f_*: H_n(X; \mathbb{Z}) \xrightarrow{\cong} H_n(Y; \mathbb{Z})$$
is an isomorphism for all $n \ge 0$, then $f$ is a homotopy equivalence.
\end{theorem}

The following, continuous version of the Alexander duality gives us a correspondence between the (reduced) homology and cohomology of a space $X$ and its dual. 
\begin{theorem}[Alexander duality,\cite{Hatcher2002}, Corollary~3.45]\label{thm:duality}
Let $X\subseteq S^n$ be a finite CW complex that is compact, locally contractible, nonempty and proper subspace of $S^n$, then 
\[
\tilde{H}_i(S^n \setminus X;\mathbb{Z}) \cong \tilde{H}^{n-i-1}(X;\mathbb{Z}).
\]
\end{theorem}
As a corollary, for geometric realizations of simplicial complexes we obtain 
\begin{corollary}\label{cor:AlexanderDual}
    Let $K$ be a simplicial complex on the vertex set $V$ such that $|V| = d$. Then  $\tilde{H}_i(\gr{K^\AD};\mathbb{Z}) \cong \tilde{H}^{d-i-3}(\gr{K};\mathbb{Z})$.
\end{corollary}

To simplify even further, we will be using Corollary~\ref{cor:AlexanderDual} only in the special case when $\gr{K} $ has the homotopy type or a wedge sum of spheres $\gr{K} \simeq \bigvee_{i = 1}^{k} S^n$. We remark that in that case 
\[
\tilde{H}_j(\gr{K})  \cong \tilde{H}^{j}(\gr{K};Z) \cong 
\begin{cases}
    \bigoplus_{i=1}^{k} \mathbb{Z} & \text{if } j = n, \\
    0 & j\neq n.
\end{cases}
\]

The following notion of join of topological spaces roughly corresponds to the join of simplicial complexes in the following sense: Given simplicial complexes $K,L$, $\gr{K\star L} \simeq \gr{K}\star \gr{L}$ (see discussion below \cite[Proposition~4.2.4]{matousek2003}).
\begin{definition}
Let $X,Y$ be finite CW complexes. The join of $X$ and $Y$ denoted $X\star Y$ is defined by
\[
X \star Y = X \times Y \times [0,1]/\sim,
\] where $(x,y,1) \sim (x',y,1)$ and $(x,y,0) \sim (x,y',0)$ for all $x,x' \in X $ and $y,y' \in Y$.

\end{definition}
Further a join $X \star *$ with a point $*$ is called a \emph{cone} and denoted $CX$ and that join of $X$ with two disjoint points (or gluing of two cones $CX$ along their bases) is the suspension $\Sigma X$ (see \cite[Chapter~0]{Hatcher2002} for details).

It is classical that the join of spheres $S^k \star S^\ell$ is a sphere $S^{k + \ell + 1}$, assuming $k,l\geq 0$ (see e.g. \cite[Example~4.2.2]{matousek2003}). We formulate a slight generalization of this result and omit the proof, just noting that it can be easily deduced by e.g. using the \emph{reduced join} construction.

\begin{lemma}\label{l:joinSpheres}
Let $K, L$ be simplicial complexes, $m,n \in \mathbb{N}$ and $0\leq k,l \in \mathbb{Z}$ such that $\gr{K} \simeq \bigvee_{j=1}^n S^k$ and $\gr{L} \simeq \bigvee_{i=1}^m S^\ell$. Then
\[
\gr{K} \star \gr{L} \simeq \bigvee_{i=1}^{mn} S^{k + \ell + 1}.
\]
\end{lemma}

The following result is essentially described in \cite[Example~0.14 and Example 0.13]{Hatcher2002}.
\begin{lemma}\label{l:contraction}\label{l:contractionSuspension}
Let $A\subseteq K$, $L$ be simplicial complexes.
\begin{itemize}
    \item If $A\subseteq K$ is contractible, then $\gr{K}/\gr{A}\simeq \gr{K}$.
    \item If $A\subseteq K$ is contractible in $K$, then $\gr{K \cup CA}\simeq \gr{K} \vee \gr{SA}$.
    \item if $A\subseteq L $ and $A$ is contractible. Then $\gr{K \cup L}\simeq \gr{K} \vee \gr{L}$.
\end{itemize}
\end{lemma}

The following theorem is a classical result \cite[Corollary 3.1]{bjornerWW2005} that allows us to deduce homotopy equivalence of two complexes via a simplicial map between them.
\begin{theorem}\label{thm:quillenfiber}
    Let $K, L$ be simplicial complexes and let $p\colon K \to L$ be a simplicial map. If  $p^{-1}(\sigma) = \{ \tau\mid p(\tau) \subseteq \sigma\} $ is contractible for all $\sigma\in L$, then $\gr{K} \simeq \gr{L}$.
\end{theorem}
The following lemma states that the embedded join of two complexes is homotopy equivalent to their join under the condition that their intersection is a simplex:
\begin{lemma}\label{l:joinsimplex}
Let $K,L$ be simplicial complexes and let their intersection $K \cap L$ be the standard simplex $\Delta^n$ for some $n\geq 0$. Then    
\begin{align*}
\gr{K \oplus L} \simeq \gr{K \star L}.
\end{align*}
\end{lemma}
\begin{proof}
The proof follows from Theorem~\ref{thm:quillenfiber} by considering the simplicial map $p\colon K \star L 
\to K \oplus L$ defined as $p(\sigma \times \{0\} \cup \tau \times \{1\}) = \sigma \cup \tau$. One can easily verify that for each simplex $\gamma \in K \oplus L$, the preimage $p^{-1}(\gamma)$ is isomorphic to a simplex, thus contractible. The claim follows.
\end{proof}

\section{Proofs of main results}\label{s:Proofs}
We begin with the following easy Lemma, that directly implies Theorem~\ref{thm:main} for $k = 2$. Informally, the number of $1$-dimensional "holes" in the clique complex is the number of squares in the square sequence:
\begin{lemma}\label{l:clique}
Let $G$ be a square sequence graph of length $n$. Then $
\Cliq{2}{G} \simeq \bigvee_{n} S^{1}$.
\end{lemma}
\begin{proof}
 Since $G$ is a square sequence graph, it is triangle-free, and thus $\Cliq{2}{G} = G$ if it is interpreted as a $1$-dimensional simplicial complex. We claim that every $H_i$, $n\geq i\geq 1$ in the square sequence of $G$ contains a spanning tree $T_i$ such that $|E(H_i)\setminus E(T_i)| = i$. We leave the proof of this fact to the reader.

Then $T=T_n$ is a contractible subcomplex of $\Cliq{2}{G} = G = H_n$. Lemma~\ref{l:contraction} gives
\[
\gr{\Cliq{2}{G}} \simeq \gr{\Cliq{2}{G}}/ \gr{T} \simeq \bigvee_{n} S^{1}.
\]
\end{proof}

We are now ready to prove Theorem~\ref{thm:main2} using the decomposition given in Lemma~\ref{l:decomposition} and a number of topological observations:

\begin{proof}[Proof of Theorem~\ref{thm:main2}]
Let $k=3$, since for $k = 2$ the result follows from Lemma~\ref{l:clique}. We further remark that $\Cliq{3}{C_4} \cong \Delta^3$ and $\gr{\Cliq{3}{C_4}} \simeq * \simeq \bigvee_0 S^3$. This gives us the base case for the induction.

We assume that the statement holds for all square sequence graphs $H_i$ of length $i<n$ and we will prove that it also holds for $H_{n}$.

Let $K = \Cliq{3}{H_{n-1}}$ and $L = \Cliq{2}{H_{n-1}} \oplus \Cliq{2}{G_{n}}$. Then by Lemma~\ref{l:decomposition}, we have the decomposition $\Cliq{3}{H_{n}} = K \cup L$ and $
K \cap L = \Cliq{2}{H_{n-1}} \oplus \Cliq{2}{G'_{n}}$, 
where $G'_{n} = G_{n}[V(G_{n}) \setminus \{x\}]$ in case of corner gluing and $G'_{n} = G_{n}[ V(G_{n}) \setminus \{x,v\}]$ in case of edge gluing. 

The rest of the proof is split into two cases depending on how $G_{n}$ is attached to $H_{n-1}$:
\vspace{1ex}
\newline
\textbf{Case 1: $G_{n}$ is attached to $H_{n-1}$ by edge gluing:}
By induction assumption $\gr{K} \simeq \bigvee_{\gamma_{n-1}} S^3$ and we will now describe the homotopy types of $\gr{L}$ and $\gr{K \cap L}$:

Since $L = \Cliq{2}{H_{n-1}} \oplus \Cliq{2}{G_{n}}$ and the intersection $\Cliq{2}{H_{n-1}} \cap \Cliq{2}{G_{n}}$ is the edge ($1$-simplex) $\{u,y\}$, Lemma~\ref{l:joinsimplex} implies that $\gr{L} \simeq \gr{\Cliq{2}{H_{n-1}} \star \Cliq{2}{G_{n}}}$. 

Further, by Lemma~\ref{l:clique}, $\gr{\Cliq{2}{H_{n-1}}} \simeq \bigvee_{n-1} S^1$ and $\gr{\Cliq{2}{G_{n}}} \simeq S^1$. 
Thus by Lemma~\ref{l:joinSpheres}, $\gr{L} \simeq \bigvee_{{n-1}} S^{1} \star S^1 \simeq \bigvee_{{n-1}} S^{3}$.

The intersection $K \cap L = \Cliq{2}{H_{n-1}} \oplus \Cliq{2}{G'_{n}}$ is a contractible complex by Lemma~\ref{l:joinsimplex} since $G'_{n}$ is an edge. Hence $\gr{K \cap L} \simeq *$ and Lemma~\ref{l:contraction} implies that $\gr{K \cup L} \simeq \gr{K} \vee \gr{L}$.

In summary, we get
\[\gr{\Cliq{3}{H_n}} \simeq \bigvee_{\gamma_{n-1}} S^3 \vee \bigvee_{n-1} S^3 \simeq \bigvee_{\gamma_n = {\gamma_{n-1} + {n-1}}} S^3.\]
\newline
\vspace{1ex}
\textbf{Case 2: $G_{n}$ is attached to $H_{n-1}$ by corner gluing:}
We want to argue similarly as in Case~1, by first decomposing $\Cliq{3}{H_n} = K \cup L$ into $K = \Cliq{3}{H_{n-1}}$ and $L = \{\sigma \in \Cliq{3}{H_n} \mid \sigma \cup \{x\} \in \Cliq{3}{H_n}\}$ and then describing the homotopy types of $K$,$L$ and $K \cap L$:

Following from the definition of $L$, we get $L \cong (\Cliq{2}{H_{n-1}} \oplus \{u,v\} )\oplus \{x\}$. Using the notation $L_u = \Cliq{2}{H_{n-1}} \oplus \{u\}$ and $L_v = \Cliq{2}{H_{n-1}} \oplus \{v\}$, we further rewrite as $L = (L_u \cup L_v) \star x$ , thus $L \cong C((L_u \cup L_v))$. Observe that $K\cap L = L_u \cup L_v$. We will show that $K\cap L$ is contractible in $K$, thus we will get that $\gr{K \cup L} \simeq \gr{K} \vee S(\gr{L_u \cup L_v})$ via Lemma~\ref{l:contractionSuspension} and then we compute the homotopy type of $\gr{L_u \cup L_v}$.

We will first give a complete description of the simplices of $L_u \cap L_v$: Let $A$ be the set of vertices $w$ of $H_{n-1}$ that share an edge to both $u$ and $v$ in $H_{n-1}$, i. e. 
\[A = \{w \in V(H_{n-1}) \mid \{u,w\}, \{v,w\} \in E(H_{n-1})\}.\]
We notice that $A$ is always non-empty\footnote{For illustration in Figure~\ref{fig:attach}, the size of $A$ is $2$.}, since $H_{n-1}$ contains the vertex $y$ that is connected to both $u$ and $v$. Further, $H_{n-1}[A \cup \{u,v\}]$ is a square sequence graph of length $|A|-1$ and thus $\gr{\Cliq{2}{H_{n-1}[A \cup \{u,v\}]}} \simeq \bigvee_{|A|-1} S^1$ by Lemma~\ref{l:clique}.

The simplices of $L_u \cap L_v$ are given by
\[
L_u \cap L_v  = (\Cliq{2}{H_i} \oplus \{u\}) \cap (\Cliq{2}{H_i} \oplus \{v\}) = \Cliq{2}{H_i} \cup \{\{u,v,w\} \mid w \in A\}.
\]
The simplices $\{\{u,v,w\} \mid w \in A\}$ provide a filler of the cycles in $\Cliq{2}{H_{n-1}[A \cup \{u,v\}]}$. Since these are also cycles in $\Cliq{2}{H_{n-1}}$, $L_u \cap L_v$ is homotopy equivalent to the complex obtained from $\Cliq{2}{H_{n-1}}$ by filling in all cycles that are in $\Cliq{2}{H_{n-1}[A \cup \{u,v\}]}$. Thus, by contracting these cycles, we obtain that
\[
\gr{L_u \cap L_v} \simeq \bigvee_{{n-1} - |A| + 1} S^1.
\]

Since $L_u$ is (up to homotopy) a cone over $L_u \cap L_v$ with apex $u$ (see Lemma~\ref{l:joinsimplex}) and similarly $L_v$ is a cone over $L_u \cap L_v$ with apex $v$, it follows that $L_u, L_v$ are contractible. We want to further show that $\gr{L_u \cup L_v} \simeq \gr{\Sigma(L_u\cap L_v)}$. To this end\footnote{For readers more accustomed to homotopy theory, we can alternatively argue that $\gr{L_u \cap L_v}$ is a homotopy pushout of $\gr{L_u} \simeq * $ and $\gr{L_v} \simeq *$ over $\gr{L_{u}\cap L_v}$, and since $\gr{L_{u}\cap L_v}$ is contractible in both $\gr{L_u}$ and $\gr{L_v}$, the result follows.}, it suffices to consider the simplicial map $p\colon \Sigma(L_u\cap L_v) \to (L_u \cup L_v)$ which is an identity on vertices of $L_u \cap L_v$ and apexes ${a_1, a_2}$ of the suspension are mapped to $u$ and $v$, respectively. One now uses Theorem~\ref{thm:quillenfiber} and checks that $p$ indeed induces a homotopy equivalence, which is left to the reader. In summary, this gives us
\[\gr{L_u \cup L_v} \simeq \Sigma(\bigvee_{{n-1} - |A| + 1} S^1) \simeq \bigvee_{{n-1} - |A| + 1} S^2.\]
However, $\gr{K}$ is by induction assumption homotopy equivalent to $\bigvee_{\gamma_{n-1}} S^3$ and thus $K \cap L = L_u \cup L_v$ is contractible in $K$. Moreover, $L$ is a cone over $L_u \cup L_v$ with apex $x$, and Lemma \ref{l:contractionSuspension} gives 
\[\gr{K \cup L} \simeq \gr{K} \vee \Sigma(\gr{L_u \cup L_v}) \simeq \bigvee_{\gamma_{n-1}} S^3 \vee \bigvee_{{n-1} - |A| + 1} S^3 \simeq \bigvee_{\gamma_n = \gamma_{n-1} + {n-1} - |A| + 1} S^3.\]
\end{proof}

We now use the preceding result and prove Theorem~\ref{thm:main}.
\begin{proof}[Proof of Theorem~\ref{thm:main}]
For $k = 2$, the result follows from Lemma~\ref{l:clique}. 
We will be using the square sequence
\[
G_1 = H_1 \subseteq H_2 \subseteq \cdots \subseteq H_{(m-1)(n-1)} \cong G
\]
of $G = G_{m,n}$, $m,n \geq 2$ that can be informally described as "row by row" and "bottom up". Formally, let $H'_{i,j}$, $1\leq i<m,1\leq j<n$ denote the subgraph of $G_{m,n}$ induced by the vertices $\{(x,y) \mid x\in\{i,i+1\}, y \in \{j,j+1 \}\}$. Then the square sequence of $G_{m,n}$ is defined as follows:

\[
G_i = H'_{\floor{\frac{i}{m-1}} + 1 , ((i-1) \% (m-1)) + 1}, \quad 1\leq i\leq (m-1)(n-1).
\]

It follows that in this sequence we are edge gluing  the squares in the first row and every leftmost square in each row and corner gluing the rest. 
Importantly, we observe that whenever we are corner gluing to some $H_i$, the set $A$ 
\[A = \{w \in V(H_{i}) \mid \{u,w\}, \{v,w\} \in E(H_{i})\}.\]
is of size $1$. 

We now prove that $\gr{\Cliq{3}{H_i}} \simeq \bigvee_{\binom{i}{2}} S^3$ by induction using Theorem~\ref{thm:main2}. The basic step is provided by Lemma~\ref{l:clique} and the fact that $\gr{\Cliq{3}{H_1}} \simeq *$. 
Thus we assume the result holds for all $1\leq i< (m-1)(n-1)$ and we prove it for $i+1$: In case $H_{i+1}$ is obtained via edge gluing, we get by the first case of the proof of Theorem~\ref{thm:main2} that 
\[
\gr{\Cliq{3}{H_i+1}} = \gr{\Cliq{3}{H_{i}}} \vee \bigvee_{i} S^3 \simeq \bigvee_{\binom{i}{2}} S^3 \vee \bigvee_{\binom{i}{1}} S^3 \simeq \bigvee_{\binom{i+1}{2}} S^3.
\]

In case $H_{i+1}$ is obtained via corner gluing, we get by the second case of the proof of Theorem~\ref{thm:main2} that 
\[
\gr{\Cliq{3}{H_i+1}} = \gr{\Cliq{3}{H_{i}}} \vee \bigvee_{i - |A| + 1} S^3 \simeq \bigvee_{\binom{i}{2}} S^3 \vee \bigvee_{\binom{i}{1}} S^3 \simeq \bigvee_{\binom{i+1}{2}} S^3,
\]
which finishes the proof.
\end{proof}
Again using details from the proof of Theorem~\ref{thm:main2}, we obtain the proof of Corollary~\ref{c:edgeseqeunce}.
\begin{proof}[Proof of Corollary~\ref{c:edgeseqeunce}]
The proof is by induction with respect to $n$ and uses the fact that edge gluing formula in the Case~1 of the proof of Theorem~\ref{thm:main2} holds for general $k$:
The basic step is again provided by Lemma~\ref{l:clique} and the simple observation that $\Cliq{k}{H_i}$ is contractible if $i<k$.
Next, we assume the statement holds for square sequence graphs of length $i<n$ and we will prove it for $n$. To this end, we essentially rewrite the proof of Theorem~\ref{thm:main2}, Case~1:

First, we decompose $\Cliq{k}{H_n}$ using the decomposition Lemma~\ref{l:decomposition} as a union of $K = \Cliq{k}{H_{n-1}}$ and $L = \Cliq{k-1}{H_{n-1}} \oplus \Cliq{2}{G_{n}}$.
By induction, $\gr{K}$ is homotopy equivalent to $\bigvee_{\binom{n-2}{k-1}} S^{2k-3}$. Also by induction and the fact we are edge gluing, we obtain
\[
\gr{L} \simeq \bigvee_{\binom{n-2}{k-2}} S^{2{(k-1)} - 3} \star S^1 \simeq \bigvee_{\binom{n-2}{k-2}} S^{2k - 3}.
\]
The intersection of $K\cap L$ is contractible and is contractible in both $K$ and $L$ from which we get that 
\[
\gr{\Cliq{k}{H_n}} \simeq \gr{K\cup L} \simeq \bigvee_{\binom{n-2}{k-1}} S^{2k-3} \vee \bigvee_{\binom{n-2}{k-2}} S^{3} \simeq \bigvee_{\binom{n-1}{k-1}} S^{2k-3}
\]
by the binomial formula.
\end{proof}
Finally, we proceed to the proof of Corollary~\ref{thm:TotalCutMain}:
\begin{proof}[Proof of Corollary~\ref{thm:TotalCutMain}]
For $k = 2$, the statement was proved in~\cite[Theorem~4.16]{BayerDJRSX24} and further for $k=3$ and in the case one of $m,n$ equals $2$ or $3$ it was given in~\cite[Theorem~1.1]{ChandrakarHRS24}.

So let us assume that $m>n \geq 3$. By
Theorem~\ref{thm:main}, we have
\[
 \gr{\Cliq{3}{G_{m,n}}} \simeq \bigvee_{
    \binom{(m-1)(n-1)}{2}} S^{3},
\]
The Alexander duality (Corollary~\ref{cor:AlexanderDual}) implies that the homology and cohomology groups of $\gr{\Delta_3^t (G)}$ are concentrated in dimension $mn - 6$ and are equal to $\bigoplus_{\binom{(m-1)(n-1)}{2}} \mathbb{Z}$. We thus only have to verify that $\pi_1(\gr{\Delta_3^t (G_{m,n})}) = 0$ and the rest follows from the combination of theorems of Whitehead and Hurewicz (Theorem~\ref{t:whitehead}).

It suffices to show that 
\[
\Delta_3^t (G_{m,n}) = \{ W\subseteq V(G_{m,n}) \mid (V(G_{m,n})\setminus W) \supseteq \sigma \in I_3(G_{m,n}) \}
\]
contains all $2$-simplices = subsets of vertices of size $3$, since then all cycles have a filler. However, in case $m,n \geq 3$, this condition is trivially satisfied: the graph $G_{3,3}$ is bipartite with one part containing $4$ and other $5$ vertices, so for any set of vertices $W$ of size $3$, $V(G)\setminus W$ contains an independent set of size $3$. So we can see that for $m,n\geq 3$, $\pi_1(\gr{\Delta_3^t (G_{m,n})}) = 0$ and thus 
\[
 \gr{\Delta_3^t (G_{m,n})} \simeq \bigvee_{
    \binom{(m-1)(n-1)}{2}} S^{3}.
\]
\end{proof}

\subparagraph*{Acknowledgements.}
I'd like to thank Lukáš Vokřínek for being a sounding board and discussion partner when preparing this article, Marija Jeli\'{c} Milutinovi\'{c} and anonymous reviewers for their comments on the preliminary version of this article.





\bibliography{references}
\end{document}